\documentclass[a4paper,11pt]{article}
\usepackage{graphicx,epsfig,picins}
\usepackage{fancyhdr,fancybox}
\usepackage{indentfirst}
\usepackage{titlesec}
\usepackage{verbatim}
\usepackage[sort&compress, numbers]{natbib}
\usepackage{array,dcolumn,tabularx}
\usepackage{setspace}
\usepackage{geometry}
\usepackage{extarrows,chemarrow,xypic} 
\usepackage[small]{caption2}
\usepackage{sectsty}
\usepackage{microtype}
\DisableLigatures[f]{encoding = *, family = * }

\usepackage{times}     

\usepackage{latexsym}
\usepackage{amsmath}   
\usepackage{amssymb}   
\usepackage{amsbsy}
\usepackage{amsthm}
\usepackage{amsfonts}
\usepackage{mathrsfs}  
\usepackage{bm}        
\usepackage{relsize}   
\usepackage{hyperref}  

\newcommand{\paperfont}{\fontsize{12pt}{1.3\baselineskip}\selectfont}
\sectionfont{\fontsize{13}{15}\selectfont}
\begin{document}


\theoremstyle{definition}
\makeatletter
\thm@headfont{\bf}
\makeatother
\newtheorem{definition}{Definition}
\newtheorem{example}{Example}
\newtheorem{theorem}{Theorem}
\newtheorem{lemma}{Lemma}
\newtheorem{corollary}{Corollary}
\newtheorem{remark}{Remark}
\newtheorem{case}{Case}
\newtheorem{proposition}{Proposition}

\lhead{}
\rhead{}
\lfoot{}
\rfoot{}

\renewcommand{\refname}{References}
\renewcommand{\figurename}{Figure}
\renewcommand{\tablename}{Table}
\renewcommand{\proofname}{Proof}

\newcommand{\diag}{\mathrm{diag}}
\newcommand{\tr}{\mathrm{tr}}

\title{\textbf{A solution to the reversible embedding problem for finite Markov chains}}
\author{Chen Jia$^{1,2}$ \\
\footnotesize $^1$Beijing Computational Science Research Center, Beijing 100094, P.R. China. \\
\footnotesize $^2$Department of Mathematical Sciences, The University of Texas at Dallas, Richardson, Texas 75080, U.S.A. \\
\footnotesize Email: jiac@utdallas.edu}
\date{}                              
\maketitle                           
\thispagestyle{empty}                

\paperfont

\begin{abstract}
The embedding problem for Markov chains is a famous problem in probability theory and only partial results are available up till now. In this paper, we propose a variant of the embedding problem called the reversible embedding problem which has a deep physical and biochemical background and provide a complete solution to this new problem. We prove that the reversible embedding of a stochastic matrix, if it exists, must be unique. Moreover, we obtain the sufficient and necessary conditions for the existence of the reversible embedding and provide an effective method to compute the reversible embedding. Some examples are also given to illustrate the main results of this paper. \\

\end{abstract}

\section{Introduction}
In 1937, Elfving \cite{elfving1937theorie} proposed the following problem: given an $n\times n$ stochastic matrix $P$, can we find an $n\times n$ generator matrix $Q$ such that $P = e^Q$? This problem, which is referred to as the embedding problem for stochastic matrices or the embedding problem for finite Markov chains, is still an open problem in probability theory. Let $X = \{X_n:n\geq 0\}$ be a discrete-time homogeneous Markov chain with transition probability matrix $P$. The embedding problem is equivalent to asking whether we can find a continuous-time homogeneous Markov chain $Y = \{Y_t:t\geq 0\}$ with transition semigroup $\{P(t):t\geq 0\}$ such that $P = P(1)$. If this occurs, the discrete-time Markov chain $X$ can be embedded as the discrete skeleton of the continuous-time Markov chain $Y$.

The embedding problem has been studied for a long time \cite{kingman1962imbedding, runnenburg1962elfving, chung1967markov, speakman1967two, cuthbert1973logarithm, johansen1974some, singer1975social, singer1976representation, carette1995characterizations, davies2010embeddable, chen2011imbedding, guerry2013embedding}. So far, the embedding problem for $2\times 2$ stochastic matrices has been solved by Kendall and this result is published in Kingman \cite{kingman1962imbedding}. The embedding problem for $3\times 3$ stochastic matrices has also been solved owing to the work of Johansen \cite{johansen1974some}, Carette \cite{carette1995characterizations}, and Chen \cite{chen2011imbedding}. However, when the order $n$ of the stochastic matrix $P$ is larger than three, only partial results are available and our knowledge on the set of embeddable $n\times n$ stochastic matrices is quite limited. There is also an embedding problem for inhomogeneous Markov chains which has been dealt with by some authors \cite{goodman1970intrinsic, johansen1973central, frydman1979total, johansen1979bang, frydman1980embedding, frydman1980structure, frydman1983number, fuglede1988imbedding, lencastre2016empirical}. However, we only focus on the homogeneous case in this paper.

The embedding problem has wide applications in many scientific fields, such as social science \cite{singer1975social, singer1976representation}, mathematical finance \cite{israel2001finding}, statistics \cite{bladt2005statistical, metzner2007generator}, biology \cite{verbyla2013embedding, jia2014overshoot}, and manpower planning \cite{guerry2014some}. One of the most important applications of the embedding problem is the generator estimation problem in statistics. Let $Y = \{Y_t:t\geq 0\}$ be a continuous-time Markov chain with generator matrix $Q$. In practice, it often occurs that we can only observe a sufficiently long trajectory of $Y$ at several discrete times $0,T, 2T,\cdots,mT$ with time interval $T$. Let $P = (p_{ij})$ be the transition probability matrix of $Y$ at time $T$. Then $p_{ij}$ can be estimated by the maximum likelihood estimator
\begin{equation}
\hat{p}_{ij} = \frac{\sum_{k=0}^{m-1}I_{\{Y_{kT}=i,Y_{(k+1)T}=j\}}}{\sum_{k=0}^{m-1}I_{\{Y_{kT}=i\}}},
\end{equation}
where $I_A$ denotes the indicator function of $A$. A natural question is whether we can obtain an estimator $\hat{Q}$ of the generator matrix $Q$ from the estimator $\hat{P} = (\hat{p}_{ij})$ of the transition probability matrix $P$. It is reasonable to require the estimators $\hat{P}$ and $\hat{Q}$ to be related by $\hat{P} = e^{\hat{Q}T}$. Therefore, the generator estimation problem in statistics is naturally related to embedding problem for finite Markov chains.

Recently, the generator estimation problem has been widely studied in biology \cite{verbyla2013embedding, jia2014overshoot}, since a number of biochemical systems can be modeled by continuous-time Markov chains. In general, there are two types of Markov chains that must be distinguished, reversible chains and irreversible chains. In the reversible case, the detailed balance condition $\mu_iq_{ij} = \mu_jq_{ji}$ holds for any pair of states $i$ and $j$, where $\mu_i$ is the stationary probability of state $i$ and $q_{ij}$ is the transition rate from state $i$ to $j$ \cite{norris1998markov}. From the physical perspective, detailed balance is a thermodynamic constraint for closed systems. In other words, if there is no sustained energy supply, then a biochemical system must satisfy detailed balance \cite{qian2007phosphorylation}. In the modelling of many biochemical systems such as enzymes \cite{cornish2012fundamentals} and ion channels \cite{sakmann2009single}, detailed balance has become a basic assumption \cite{alberty2004principle}. Therefore, in many realistic biochemical systems, what we are concerned about is not simply to find a generator matrix $\hat{Q}$ such that $\hat{P} = e^{\hat{Q}T}$, but to find a reversible Markov chain with generator matrix $\hat{Q}$ such that $\hat{P} = e^{\hat{Q}T}$.

Here we consider the following problem: given an $n\times n$ stochastic matrix $P$, can we find a reversible generator matrix $Q$ such that $P = e^Q$? This problem will be referred to as the reversible embedding problem for stochastic matrices or the reversible embedding problem for finite Markov chains in this paper. Compared with the classical embedding problem, the reversible embedding problem has a deeper physical and biochemical background.

In this paper, we provide a complete solution to the reversible embedding problem. We prove that the reversible embedding of stochastic matrices, if it exists, must be unique. Moreover, we give the sufficient and necessary conditions for the existence of the reversible embedding and provide an effective method to compute the reversible embedding. Finally, we use some examples of $3\times 3$ stochastic matrices to illustrate the main results of this paper.

\section{Preliminaries}
For clarity, we recall several basic definitions.
\begin{definition}
An $n\times n$ real matrix $P=(p_{ij})$ is called a stochastic matrix if $p_{ij}\geq 0$ for any $i,j = 1,2,\cdots n$ and $\sum_{j=1}^np_{ij} = 1$ for any $i = 1,2,\cdots n$.
\end{definition}

\begin{definition}
An $n\times n$ real matrix $Q=(q_{ij})$ is called a generator matrix if $q_{ij}\geq 0$ for any $i\neq j$ and $\sum_{j=1}^nq_{ij} = 0$ for any $i = 1,2,\cdots n$.
\end{definition}

In this paper, we consider a fixed $n\times n$ stochastic matrix $P$. For simplicity, we assume that $P$ is irreducible. Otherwise, we may restrict our discussion to an irreducible recurrent class of $P$. Since $P$ is irreducible, it has a unique invariant distribution $\mu = (\mu_1,\mu_2,\cdots,\mu_n)$ whose components are all positive \cite{norris1998markov}.

\begin{definition}
$P$ is called reversible if the detailed balance condition $\mu_ip_{ij} = \mu_jp_{ji}$ holds for any $i,j = 1,2,\cdots,n$. In this case, $\mu$ is called a reversible distribution of $P$.
\end{definition}

\begin{definition}
Let $Q$ be an $n\times n$ generator matrix. Then $Q$ is called reversible if there exists a distribution $\pi = (\pi_1,\pi_2\cdots,\pi_n)$ such that the detailed balance condition $\pi_iq_{ij} = \pi_jq_{ji}$ holds for any $i,j = 1,2,\cdots,n$. In this case, $\pi$ is called a reversible distribution of $Q$.
\end{definition}

In fact, if $\pi$ is a reversible distribution of $Q$, then $\pi$ is also an invariant distribution of $Q$.

\begin{definition}
If there exists an $n\times n$ real matrix $A$ such that $P = e^A$, then $A$ is called a real logarithm of $P$.
\end{definition}

\begin{definition}
$P$ is called embeddable if there exists an $n\times n$ generator matrix $Q$ such that $P = e^Q$. In this case, $Q$ is called an embedding of $P$.
\end{definition}

It is easy to see that if $Q$ is a embedding of $P$, then $Q$ is also a real logarithm of $P$.

\section{Results}
In this paper, we shall address the reversible embedding problem for stochastic matrices. We first give the definition of reversible embeddability for stochastic matrices.

\begin{definition}
$P$ is called reversibly embeddable if there exists an $n\times n$ reversible generator matrix $Q$ such that $P = e^{Q}$. In this case, $Q$ is called a reversible embedding of $P$.
\end{definition}

\begin{lemma}
If $Q$ is a reversible embedding of $P$, then $Q$ is irreducible and $\mu$ is the reversible distribution of $Q$.
\end{lemma}

\begin{proof}
If $Q$ has two or more communicating classes, it is easy to see that $P = e^Q$ also has two or more communicating classes, which contradicts the irreducibility of $P$. This shows that $Q$ is irreducible.

Since $Q$ is reversible, it has a reversible distribution $\pi = (\pi_1,\pi_2\cdots,\pi_n)$ which is also the invariant distribution of $Q$. Therefore, $\pi$ is the unique invariant distribution of $P = e^Q$. This shows that $\pi = \mu$.
\end{proof}

\begin{remark}
Let $X = \{X_n:n\geq 0\}$ be a stationary discrete-time Markov chain on the finite state space $S = \{1,2,\cdots,n\}$ with transition probability matrix $P = (p_{ij})$. According to the above lemma, if $P$ is reversibly embeddable with reversible embedding $Q$, then there exists a reversible continuous-time Markov chain $Y = \{Y_t:t\geq 0\}$ on the state space $S$ with generator matrix $Q$ such that $X$ and $\{Y_n: n\geq 0\}$ have the same distribution. Therefore, the stationary discrete-time Markov chain $X$ can be embedded as the discrete skeleton of the reversible continuous-time Markov chain $Y$. That is why the problem dealt with in this paper is called the reversible embedding problem for finite Markov chains.
\end{remark}

In this paper, a diagonal matrix with diagonal entries $a_1,a_2,\cdots,a_n$ will be denoted by $\diag(a_1,a_2,\cdots,a_n)$.

\begin{lemma}\label{basic}
If $P$ is reversible, then $P$ is diagonalizable and the eigenvalues of $P$ are all real numbers.
\end{lemma}

\begin{proof}
Let $M = \diag(\sqrt{\mu_1},\sqrt{\mu_2},\cdots,\sqrt{\mu_n})$. Since $\mu$ is a reversible distribution of $P$, it is easy to see that $S = MPM^{-1}$ is a symmetric matrix. Thus $S$ must be diagonalizable and the eigenvalues of $S$ are all real numbers. This shows that $P = M^{-1}SM$ is also diagonalizable and the eigenvalues of $P$ are all real numbers.
\end{proof}

\begin{lemma}\label{assumptions}
If $P$ is reversibly embeddable, then $P$ is reversible and the eigenvalues of $P$ are all positive.
\end{lemma}

\begin{proof}
Since $P$ is reversibly embeddable, there exists a reversible generator matrix $Q$ such that $P = e^Q$. Let $M = \diag(\sqrt{\mu_1},\sqrt{\mu_2},\cdots,\sqrt{\mu_n})$. Since $\mu$ is a reversible distribution of $Q$, it is easy to see that $S = MQM^{-1}$ is a symmetric matrix. Thus $e^S = MPM^{-1}$ is also a symmetric matrix. This shows that $P$ is reversible. Since $Q$ is reversible, the eigenvalues of $Q$ are all real numbers. Thus the eigenvalues of $P = e^Q$ are all positive.
\end{proof}

Due to the above lemma, we always assume that $P$ is reversible and the eigenvalues of $P$ are all positive in the sequel. Let $\lambda_1,\lambda_2,\cdots,\lambda_n$ be all the eigenvalues of $P$, where $\lambda_i > 0$ for any $i = 1,2,\cdots n$. Let $\gamma_1,\gamma_2,\cdots,\gamma_m$ be the mutually different eigenvalues of $P$, where $m\leq n$. Let $n_i$ be the multiplicity of the eigenvalue $\gamma_i$ for any $i = 1,2,\cdots,m$. It is easy to see that $n_1+n_2+\cdots+n_m = n$. Let
\begin{equation}
D = \diag(\lambda_1,\lambda_2,\cdots,\lambda_n) = \diag(\gamma_1,\cdots,\gamma_1,\gamma_2,\cdots,\gamma_2,\cdots,\gamma_m,\cdots,\gamma_m).
\end{equation}
and let
\begin{equation}
\begin{split}
\log D &= \diag(\log\lambda_1,\log\lambda_2,\cdots,\log\lambda_n) \\
&= \diag(\log\gamma_1,\cdots,\log\gamma_1,\log\gamma_2,\cdots,\log\gamma_2,\cdots,
\log\gamma_m,\cdots,\log\gamma_m).
\end{split}
\end{equation}
Since $P$ is reversible, according to Lemma \ref{basic}, there exists an $n\times n$ real invertible matrix $T$ such that $P = TDT^{-1}$. Let $H$ be the matrix defined as
\begin{equation}\label{H}
H = T\log DT^{-1}.
\end{equation}
It is easy to see that $e^H = TDT^{-1} = P$, which shows that $H$ is a real logarithm of $P$.

It has been known for a long time that the embedding of a stochastic matrix may not be unique \cite{speakman1967two}. However, this is not the case when it comes to the reversible embedding. The following theorem, which is the first main result of this paper, reveals the uniqueness of the reversible embedding. It shows that the reversible embedding of a stochastic matrix, if it exists, must be unique.

\begin{theorem}\label{uniqueness}
$P$ has at most one reversible embedding. If $P$ is reversibly embeddable, then the unique reversible embedding of $P$ must be $H$.
\end{theorem}

\begin{proof}
This theorem will be proved in Section \ref{proof1}.
\end{proof}

We next study the existence of the reversible embedding.

\begin{lemma}\label{condition}
Assume that $P$ is reversible and the eigenvalues of $P$ are all positive. Then $P$ is reversibly embeddable if and only if $H$ is a reversible generator matrix.
\end{lemma}

\begin{proof}
This lemma follows directly from Theorem \ref{uniqueness}.
\end{proof}

In general, it is difficult to determine whether $H$ is a reversible generator matrix. Thus we hope to obtain simpler sufficient and necessary conditions for a stochastic matrix being reversibly embeddable. Let $k_0,k_1,\cdots,k_{m-1}\in\mathbb{R}$ be the solution to the following system of linear equations:
\begin{equation}\label{linear}\left\{
\begin{split}
k_0+k_1\gamma_1+\cdots+k_{m-1}\gamma_1^{m-1} &= \log\gamma_1, \\
k_0+k_1\gamma_2+\cdots+k_{m-1}\gamma_2^{m-1} &= \log\gamma_2, \\
\cdots &\cdots \\
k_0+k_1\gamma_m+\cdots+k_{m-1}\gamma_m^{m-1} &= \log\gamma_m.
\end{split}\right.
\end{equation}
Since
\begin{equation}
\det\begin{pmatrix}
1 &\gamma_1 &\cdots &\gamma_1^{m-1} \\
1 &\gamma_2 &\cdots &\gamma_2^{m-1} \\
\cdots &\cdots &\cdots &\cdots \\
1 &\gamma_m &\cdots &\gamma_m^{m-1}
\end{pmatrix}
= \sum_{1\leq i<j\leq m}(\gamma_j-\gamma_i)\neq 0,
\end{equation}
the solution to the above system of linear equations exists and is unique. From \eqref{linear}, it is easy to see that
\begin{equation}
\log D = k_0I+k_1D+\cdots+k_{m-1}D^{m-1},
\end{equation}
which implies that
\begin{equation}\label{combination}
H = k_0I+k_1P+\cdots+k_{m-1}P^{m-1}.
\end{equation}

The following theorem, which is the second main result of this paper, gives the sufficient and necessary conditions for the existence of the reversible embedding.

\begin{theorem}\label{existence}
Let $P^k = (p_{ij}^{(k)})$ for any $k\geq 1$. Then $P$ is reversibly embeddable if and only if the following two conditions hold: \\
(i) $P$ is reversible and the eigenvalues of $P$ are all positive; \\
(ii) $k_1p_{ij}+k_2p_{ij}^{(2)}+\cdots+k_{m-1}p_{ij}^{(m-1)}\geq 0$ for any $i\neq j$.
\end{theorem}

\begin{proof}
This theorem will be proved in Section \ref{proof2}.
\end{proof}

From \eqref{combination}, it is easy to see that the condition (ii) in the above theorem is equivalent to saying that the off-diagonal entries of $H$ are all nonnegative.

The following classical result about the embeddability for $2\times 2$ stochastic matrices is due to Kendall and is published in Kingman \cite{kingman1962imbedding}. Interestingly, the above theorem gives a simple derivation of this classical result.

\begin{theorem}[Kendall]
Let $P$ be a $2\times 2$ irreducible stochastic matrix. Then the following four statements are equivalent: \\
(i) $P$ is embeddable; \\
(ii) $P$ is reversibly embeddable; \\
(iii) $\tr(P) > 1$; \\
(iv) $\det(P) > 0$.
\end{theorem}

\begin{proof}
It is easy to see that (iii) and (iv) are equivalent. Thus we only need to prove that (i), (ii), and (iii) are equivalent.

Assume that $Q$ is an embedding of $P$. If $Q$ has two or more communicating classes, it is easy to see that $P = e^Q$ also has two or more communicating classes. This contradicts the irreducibility of $P$. This shows that $Q$ is irreducible. It is easy to see that a $2\times 2$ irreducible generator matrix $Q$ must be reversible. Thus $Q$ is a reversible embedding of $P$. This shows that (i) and (ii) are equivalent.

We next prove that (ii) and (iii) are equivalent. Let
\begin{equation}
P = \begin{pmatrix}
p & 1-p \\
1-q & q \\
\end{pmatrix}.
\end{equation}
The irreducibility of $P$ implies that $p,q < 1$. It is easy to see that $P$ is reversible and the two eigenvalues of $P$ are $\gamma_1 = 1$ and $\gamma_2 = p+q-1 < 1$. Thus the condition (i) in Theorem \ref{existence} is equivalent to saying that $p+q > 1$. Assume that $p+q > 1$. Let $k_0$ and $k_1$ be the solution to the following system of linear equations:
\begin{equation}\left\{
\begin{split}
k_0+k_1\gamma_1 &= \log\gamma_1, \\
k_0+k_1\gamma_2 &= \log\gamma_2.
\end{split}\right.
\end{equation}
It is easy to check that
\begin{equation}
k_1 = -k_0 = \frac{\log(p+q-1)}{p+q-2}.
\end{equation}
Since $1 < p+q < 2$, it is easy to see that $k_1 > 0$. Thus the off-diagonal entries of $k_1P$ are all nonnegative. This shows that the condition (ii) in Theorem \ref{existence} holds. Thus $p+q > 1$ implies the condition (ii) in Theorem \ref{existence}. By Theorem \ref{existence}, $P$ is reversibly embeddable if and only if $p+q > 1$, that is, $\tr(P) > 1$. This shows that (ii) and (iii) are equivalent.
\end{proof}

As another application of Theorem \ref{existence}, we give some simple and direct criteria for a $3\times 3$ stochastic matrix being reversible embeddable. Let $P$ be a $3\times 3$ irreducible stochastic matrix. The Perron-Frobenius theorem \cite{berman1979nonnegative} claims that one eigenvalue of $P$ must be 1 and the absolute values of the other two eigenvalues must be less than 1.

We first consider the case where $P$ has a pair of coincident eigenvalues.

\begin{theorem}\label{coincident}
Let $P$ be a $3\times 3$ irreducible stochastic matrix with eigenvalues $1$, $\lambda$, and $\lambda$. Then $P$ is reversibly embeddable if and only if $P$ is reversible and $\lambda > 0$.
\end{theorem}

\begin{proof}
The condition (i) in Theorem \ref{existence} is equivalent to saying that $P$ is reversible and $\lambda > 0$. Assume that $\lambda > 0$. The Perron-Frobenius theorem implies that $0<\lambda<1$. Since $P$ has a pair of coincident eigenvalues, the mutually different eigenvalues of $P$ are $\gamma_1 = 1$ and $\gamma_2 = \lambda$. Let $k_0$ and $k_1$ be the solution to the following system of linear equations:
\begin{equation}\left\{
\begin{split}
k_0+k_1\gamma_1 &= \log\gamma_1, \\
k_0+k_1\gamma_2 &= \log\gamma_2.
\end{split}\right.
\end{equation}
Straightforward calculations show that
\begin{equation}
k_1 = -k_0 = \frac{\log\lambda}{\lambda-1}.
\end{equation}
Since $0<\lambda<1$, it is easy to see that $k_1>0$. Thus the off-diagonal entries of $k_1P$ are all nonnegative. This shows that the condition (ii) in Theorem \ref{existence} holds. Thus $\lambda > 0$ implies the condition (ii) in Theorem \ref{existence}. By Theorem \ref{existence}, $P$ is reversibly embeddable if and only if $P$ is reversible and $\lambda > 0$.
\end{proof}

We next consider the case where $P$ has three different eigenvalues.

\begin{theorem}\label{different}
Let $P$ be a $3\times 3$ irreducible stochastic matrix with distinct eigenvalues $1$, $\lambda$, and $\eta$. Then $P$ is reversibly embeddable if and only if the following two conditions hold: \\
(i) $P$ is reversible and $\lambda,\eta > 0$; \\
(ii) $k_1p_{ij}+k_2p_{ij}^{(2)}\geq 0$ for any $i\neq j$, where
\begin{equation}\label{expressions}
\begin{split}
k_1 &= \frac{(\eta^2-1)\log\lambda-(\lambda^2-1)\log\eta}
{(\lambda-1)(\eta-1)(\eta-\lambda)}, \\
k_2 &= \frac{(\lambda-1)\log\eta-(\eta-1)\log\lambda}
{(\lambda-1)(\eta-1)(\eta-\lambda)}.
\end{split}
\end{equation}
\end{theorem}

\begin{proof}
It is easy to see that the mutually different eigenvalues of $P$ are $\gamma_1 = 1$, $\gamma_2 = \lambda$, and $\gamma_3 = \eta$. Let $k_0,k_1,k_2$ be the solution to the following system of linear equations:
\begin{equation}\left\{
\begin{split}
k_0+k_1\gamma_1+k_2\gamma_1^2 &= \log\gamma_1, \\
k_0+k_1\gamma_2+k_2\gamma_2^2 &= \log\gamma_2, \\
k_0+k_1\gamma_3+k_2\gamma_3^2 &= \log\gamma_3.
\end{split}\right.
\end{equation}
By solving the above system of linear equations, it is easy to check that \eqref{expressions} holds. The rest of the proof follows directly from Theorem \ref{existence}.
\end{proof}

\section{Examples}
In this section, we shall use some examples of $3\times 3$ stochastic matrices to illustrate the main results of this paper. Let $P = (p_{ij})$ be a $3\times 3$ irreducible stochastic matrix. It is well-known that $P$ is reversible if and only if $p_{12}p_{23}p_{31} = p_{21}p_{32}p_{13}$. This result is a direct corollary of Kolmogorov's criterion for reversibility \cite{kelly2011reversibility}, which claims that a discrete-time Markov chain is reversible if and only if the product of the transition probabilities along each cycle and that along its reversed cycle are exactly the same.

\begin{example}
Let
\begin{equation}
P = \begin{pmatrix}
\frac{1}{6} & \frac{1}{3} & \frac{1}{2}\\
\frac{1}{2} & \frac{1}{6} & \frac{1}{3}\\
\frac{1}{3} & \frac{1}{2} & \frac{1}{6}\\
\end{pmatrix}.
\end{equation}
It is easy to check that $p_{12}p_{23}p_{31}\neq p_{21}p_{32}p_{13}$, which shows that $P$ is not reversible. Thus it follows from Theorem \ref{existence} that $P$ is not reversibly embeddable.
\end{example}

\begin{example}
Let
\begin{equation}
P = \begin{pmatrix}
\frac{1}{3} & \frac{1}{2} & \frac{1}{6}\\
\frac{1}{2} & \frac{1}{6} & \frac{1}{3}\\
\frac{1}{6} & \frac{1}{3} & \frac{1}{2}\\
\end{pmatrix}.
\end{equation}
It is easy to check that $P$ is reversible and the three eigenvalues of $P$ are
\begin{equation}
\lambda_1 = 1,\;\;\;\lambda_2 = \frac{\sqrt{3}}{6},\;\;\;\lambda_3 = -\frac{\sqrt{3}}{6}.
\end{equation}
Since $\lambda_3 < 0$, it follows from Theorem \ref{existence} that $P$ is not reversibly embeddable.
\end{example}

\begin{example}
Let
\begin{equation}
P = \begin{pmatrix}
\frac{1}{2} & \frac{2}{5} & \frac{1}{10}\\
\frac{2}{5} & \frac{2}{5} & \frac{1}{5}\\
\frac{1}{10} & \frac{1}{5} & \frac{7}{10}\\
\end{pmatrix}.
\end{equation}
It is easy to check that $P$ is reversible and the three eigenvalues of $P$ are
\begin{equation}
\lambda_1 = 1,\;\;\;\lambda_2 = \frac{3+\sqrt{7}}{10},\;\;\;\lambda_3 = \frac{3-\sqrt{7}}{10},
\end{equation}
which are all positive and mutually different. Let $k_1$ and $k_2$ be the two real numbers defined in \eqref{expressions}. Straightforward calculations show that $k_1p_{13}+k_2p_{13}^{(2)} < 0$. Thus it follows from Theorem \ref{different} that $P$ is not reversibly embeddable.
\end{example}

\begin{example}
Let
\begin{equation}
P = \begin{pmatrix}
\frac{1}{2} & \frac{1}{4} & \frac{1}{4}\\
\frac{1}{4} & \frac{1}{2} & \frac{1}{4}\\
\frac{1}{4} & \frac{1}{4} & \frac{1}{2}\\
\end{pmatrix}.
\end{equation}
It is easy to check that $P$ is reversible and the three eigenvalues of $P$ are
\begin{equation}
\lambda_1 = 1,\;\;\;\lambda_2 = \lambda_3 = \frac{1}{4},
\end{equation}
which are all positive. Since $P$ has a pair of coincident eigenvalues, it follows from Theorem \ref{coincident} that $P$ is reversibly embeddable.
\end{example}

\section{Proof of Theorem \ref{uniqueness}}\label{proof1}
For convenience, we give the following definition.

\begin{definition}\label{candidatedef}
Assume that $P$ is reversible and the eigenvalues of $P$ are all positive. Let $T$ be an $n\times n$ real invertible matrix such that $P = TDT^{-1}$. Then $H = T\log DT^{-1}$ is called a candidate of $P$.
\end{definition}

It is easy to see that if $H$ is a candidate of $P$, then $H$ is a real logarithm of $P$.

\begin{lemma}\label{Jordan}
Let $\lambda$ be a complex number and let
\begin{equation}
A = \begin{pmatrix}
\lambda & 1      &         &         \\
        & \ddots & \ddots  &         \\
        &        & \lambda & 1       \\
        &        &         & \lambda \\
\end{pmatrix}_{n\times n}.
\end{equation}
Then the Jordan canonical form of $e^A$ is
\begin{equation}
\begin{pmatrix}
e^\lambda & 1      &           &          \\
          & \ddots & \ddots    &          \\
          &        & e^\lambda & 1        \\
          &        &           & e^\lambda \\
\end{pmatrix}_{n\times n}.
\end{equation}
\end{lemma}

\begin{proof}
Let
\begin{equation}
B = \begin{pmatrix}
0 & 1      &        &   \\
  & \ddots & \ddots &   \\
  &        & 0      & 1 \\
  &        &        & 0 \\
\end{pmatrix}_{n\times n}.
\end{equation}
Then $A = \lambda I+B$. Note that $B^l = 0$ for any $l\geq n$. Thus we have
\begin{equation}
\begin{split}
e^A &= \sum_{k=0}^\infty\frac{1}{k!}(\lambda I+B)^k
= \sum_{k=0}^\infty\frac{1}{k!}\sum_{m=0}^kC_k^m\lambda^mB^{k-m}
= \sum_{m=0}^\infty\lambda^m\sum_{k=m}^{m+n-1}\frac{1}{k!}C_k^mB^{k-m} \\
&= \sum_{m=0}^\infty\lambda^m\sum_{k=0}^{n-1}\frac{1}{(k+m)!}C_{k+m}^mB^k
= \sum_{m=0}^\infty\frac{1}{m!}\lambda^m\sum_{k=0}^{n-1}\frac{1}{k!}B^k
= e^\lambda\sum_{k=0}^{n-1}\frac{1}{k!}B^k.
\end{split}
\end{equation}
In view of the above equation, it is easy to check that $(e^A-e^\lambda I)^{n-1} = e^{(n-1)\lambda}B^{n-1}\neq 0$ but $(e^A-e^\lambda I)^n = 0$. This shows that $e^\lambda$ is an $n$-fold eigenvalue of $e^A$ and the Jordan canonical form of $e^A$ is $e^\lambda I+B$.
\end{proof}

\begin{lemma}\label{candidate}
Assume that $P$ is reversible and the eigenvalues of $P$ are all positive. Then each reversible embedding of $P$ must be a candidate of $P$.
\end{lemma}

\begin{proof}
Let $Q$ be a reversible embedding of $P$. Since $Q$ is reversible, the eigenvalues of $Q$ must be all real numbers. Since $P$ is diagonalizable, it follows from Lemma \ref{Jordan} that $Q$ must be also diagonalizable and thus the Jordan canonical form of $Q$ must be $\log D$. Thus there exists an $n\times n$ real invertible matrix $T$ such that $Q = T\log DT^{-1}$. Thus $P = e^Q = TDT^{-1}$. This shows that $Q$ is a candidate of $P$.
\end{proof}

\begin{lemma}\label{unique}
Assume that $P$ is reversible and the eigenvalues of $P$ are all positive. Then the candidate of $P$ is unique.
\end{lemma}

\begin{proof}
Let $H_1$ and $H_2$ be two candidates of $P$. Thus there exist two $n\times n$ real invertible matrices $T_1$ and $T_2$ such that $H_1 = T_1\log DT_1^{-1}$ and $H_2 = T_2\log DT_2^{-1}$. This shows that
\begin{equation}
T_1DT_1^{-1} = e^{H_1} = P = e^{H_2} = T_2DT_2^{-1}.
\end{equation}
Let $S = (s_{ij}) = T_2^{-1}T_1$. It is easy to see that $SD = DS$, which implies that $s_{ij}\lambda_j = \lambda_is_{ij}$ for any $i,j = 1,2,\cdots n$. This shows that $s_{ij} = 0$ whenever $\lambda_i\neq\lambda_j$. Thus whether $\lambda_i=\lambda_j$ or $\lambda_i\neq\lambda_j$, we always have $s_{ij}\log\lambda_j = \log\lambda_is_{ij}$ for any $i,j = 1,2,\cdots n$, which is equivalent to saying that $S\log D = \log DS$. Thus we obtain that
\begin{equation}
H_1 = T_1\log DT_1^{-1} = T_2\log DT_2^{-1} = H_2.
\end{equation}
This shows that the candidate of $P$ is unique.
\end{proof}

We are now in a position to prove Theorem \ref{uniqueness}.
\begin{proof}[Proof of Theorem \ref{uniqueness}]
Assume that $P$ is reversibly embeddable. By Lemma \ref{assumptions}, $P$ is reversible and the eigenvalues of $P$ are all positive. Since $H$ is a candidate of $P$, it follows from Lemmas \ref{candidate} and \ref{unique} that the reversible embedding of $P$ must be the unique candidate $H$ of $P$. Thus $P$ has at most one reversible embedding.
\end{proof}

\section{Proof of Theorem \ref{existence}}\label{proof2}
\begin{lemma}\label{reversible}
Assume that $P$ is reversible and the eigenvalues of $P$ are all positive. Let $H = (h_{ij})$ be the unique candidate of $P$. Then $\mu_ih_{ij} = \mu_jh_{ji}$ for any $i = 1,2,\cdots,n$.
\end{lemma}

\begin{proof}
Let $M = \diag(\sqrt{\mu_1},\sqrt{\mu_2},\cdots,\sqrt{\mu_n})$. Since $\mu$ is a reversible distribution of $P$, it is easy to see that $S = MPM^{-1}$ is a symmetric matrix. Thus there exists an orthogonal matrix $R$ such that $RSR^T = D$. This shows that $RMPM^{-1}R^T = D$, or equivalently, $P = M^{-1}R^TDRM$. Thus $H = M^{-1}R^T\log DRM$, or equivalently, $MHM^{-1} = R^T\log DR$. This shows that $MHM^{-1}$ is a symmetric matrix, which implies that $\mu_ih_{ij} = \mu_jh_{ji}$ for any $i = 1,2,\cdots,n$.
\end{proof}

\begin{lemma}\label{zero}
Assume that $P$ is reversible and the eigenvalues of $P$ are all positive. Let $H$ be the unique candidate of $P$. Let 1 be the $n$-dim column vector whose components are all 1. Then $H1 = 0$.
\end{lemma}

\begin{proof}
Since $H$ is a candidate of $P$, we have $e^H1 = P1 = 1$. Let
\begin{equation}
B = \sum_{k=0}^\infty\frac{1}{(k+1)!}H^k.
\end{equation}
It follows from $e^H1 = 1$ that $BH1 = 0$. Since the Jordan canonical form of $H$ is $\log D$, the Jordan canonical form of $B$ is
\begin{equation}
\sum_{k=0}^\infty\frac{1}{(k+1)!}(\log D)^k.
\end{equation}
Simple calculations show that
\begin{equation}
\sum_{k=0}^\infty\frac{1}{(k+1)!}\lambda^k =
\begin{cases}
1, &\textrm{if $\lambda = 0$}, \\
\frac{1}{\lambda}(e^\lambda-1), &\textrm{if $\lambda\neq 0$}.
\end{cases}
\end{equation}
Thus the eigenvalues of $B$ are all nonzero, which shows that $B$ is invertible. Since $BH1 = 0$, we have $H1 = 0$.
\end{proof}

\begin{remark}
In Definition \ref{candidatedef}, Lemma \ref{unique}, and Lemma \ref{zero}, the assumption ``$P$ is reversible" can be weakened to ``$P$ is diagonalizable".
\end{remark}

We are now in a position to prove Theorem \ref{existence}.

\begin{proof}[Proof of Theorem \ref{existence}]
By Lemma \ref{assumptions}, we only need to prove that if the condition (i) holds, then $P$ is reversibly embeddable if and only if the condition (ii) holds.

Assume that the condition (i) holds. Let $H$ be the unique candidate of $P$. It follows from Lemma \ref{condition} that $P$ is reversibly embeddable if and only if $H$ is a reversible generator matrix. By Lemmas \ref{reversible} and \ref{zero}, $H$ is reversible and the sum of entries in each row of $H$ is zero. This shows that $H$ is a reversible generator matrix if and only if the off-diagonal entries of $H$ are all nonnegative. From \eqref{combination}, it is easy to see that $H$ is a reversible generator matrix if and only if the off-diagonal entries of $k_0I+k_1P+\cdots+k_{m-1}P^{m-1}$ are all nonnegative. Thus $P$ is reversibly embeddable if and only if the condition (ii) holds. This completes the proof of this theorem.
\end{proof}

\section*{Acknowledgements}
The author is grateful to Prof. Y. Chen and the anonymous reviewers for their valuable comments and suggestions on the present work.

\setlength{\bibsep}{5pt}
\small\bibliographystyle{pnas2009}
\bibliography{embedding}

\begin{thebibliography}{10}

\bibitem{elfving1937theorie}
Elfving G (1937) {Zur theorie der Markoffschen ketten}.
\newblock \emph{Acta Societas Scientiarium Fennicae Nova Series A} 2:1--7.

\bibitem{kingman1962imbedding}
Kingman J (1962) {The imbedding problem for finite Markov chains}.
\newblock \emph{Z Wahrsch Verw Gebiete} 1:14--24.

\bibitem{runnenburg1962elfving}
Runnenburg JT (1962) {On Elfving's problem of imbedding a time-discrete Markov
  chain in a time-continuous one for finitely many states I}.
\newblock \emph{Proceedings of the KNAW-Series A, Mathematical Sciences}
  65:536--548.

\bibitem{chung1967markov}
Chung KL (1967) \emph{{Markov Chains with Stationary Transition Probabilities}}
  (Springer-Verlag, Berlin), 2nd edition.

\bibitem{speakman1967two}
Speakman JM (1967) {Two Markov chains with a common skeleton}.
\newblock \emph{Z Wahrsch Verw Gebiete} 7:224--224.

\bibitem{cuthbert1973logarithm}
Cuthbert JR (1973) {The logarithm function for finite-state Markov
  semi-groups}.
\newblock \emph{J Lond Math Soc} 2:524--532.

\bibitem{johansen1974some}
Johansen S (1974) {Some results on the imbedding problem for finite Markov
  chains}.
\newblock \emph{J Lond Math Soc} 2:345--351.

\bibitem{singer1975social}
Singer B, Spilerman S (1975) Social mobility models for heterogeneous
  populations.
\newblock In \emph{Sociological Methodology 1973--1974}, ed. Costner HL
  (Jossey-Bass, San Francisco), pp. 356--401.

\bibitem{singer1976representation}
Singer B, Spilerman S (1976) {The representation of social processes by Markov
  models}.
\newblock \emph{Am J Sociol} 82:1--54.

\bibitem{carette1995characterizations}
Carette P (1995) Characterizations of embeddable 3$\times$3 stochastic matrices
  with a negative eigenvalue.
\newblock \emph{New York J Math} 1:129.

\bibitem{davies2010embeddable}
Davies E (2010) {Embeddable Markov matrices}.
\newblock \emph{Electron J Probab} 15:1474--1486.

\bibitem{chen2011imbedding}
Chen Y, Chen J (2011) {On the imbedding problem for three-state time
  homogeneous Markov chains with coinciding negative eigenvalues}.
\newblock \emph{J Theor Probab} 24:928--938.

\bibitem{guerry2013embedding}
Guerry MA (2013) {On the embedding problem for discrete-time Markov chains}.
\newblock \emph{J Appl Probab} 50:918--930.

\bibitem{goodman1970intrinsic}
Goodman GS (1970) {An intrinsic time for non-stationary finite Markov chains}.
\newblock \emph{Z Wahrsch Verw Gebiete} 16:165--180.

\bibitem{johansen1973central}
Johansen S (1973) {A central limit theorem for finite semigroups and its
  application to the imbedding problem for finite state Markov chains}.
\newblock \emph{Z Wahrsch Verw Gebiete} 26:171--190.

\bibitem{frydman1979total}
Frydman H, Singer B (1979) {Total positivity and the embedding problem for
  Markov chains}.
\newblock \emph{Math Proc Cambridge} 86:339--344.

\bibitem{johansen1979bang}
Johansen S, Ramsey FL (1979) A bang-bang representation for 3$\times$3
  embeddable stochastic matrices.
\newblock \emph{Z Wahrsch Verw Gebiete} 47:107--118.

\bibitem{frydman1980embedding}
Frydman H (1980) {The embedding problem for Markov chains with three states}.
\newblock \emph{Math Proc Cambridge} 87:285--294.

\bibitem{frydman1980structure}
Frydman H (1980) A structure of the bang-bang representation for 3$\times$3
  embeddable matrices.
\newblock \emph{Z Wahrsch Verw Gebiete} 53:305--316.

\bibitem{frydman1983number}
Frydman H (1983) {On a number of Poisson matrices in bang-bang representations
  for 3$\times$3 embeddable matrices}.
\newblock \emph{J Multivariate Anal} 13:464--472.

\bibitem{fuglede1988imbedding}
Fuglede B (1988) On the imbedding problem for stochastic and doubly stochastic
  matrices.
\newblock \emph{Probab Theory Relat Fields} 80:241--260.

\bibitem{lencastre2016empirical}
Lencastre P, Raischel F, Rogers T, Lind PG (2016) {From empirical data to
  time-inhomogeneous continuous Markov processes}.
\newblock \emph{Phys Rev E} 93:032135.

\bibitem{israel2001finding}
Israel RB, Rosenthal JS, Wei JZ (2001) {Finding generators for Markov chains
  via empirical transition matrices, with applications to credit ratings}.
\newblock \emph{Math Financ} 11:245--265.

\bibitem{bladt2005statistical}
Bladt M, S{\o}rensen M (2005) {Statistical inference for discretely observed
  Markov jump processes}.
\newblock \emph{J R Stat Soc B} 67:395--410.

\bibitem{metzner2007generator}
Metzner P, Dittmer E, Jahnke T, Sch{\"u}tte C (2007) {Generator estimation of
  Markov jump processes}.
\newblock \emph{J Comput Phys} 227:353--375.

\bibitem{verbyla2013embedding}
Verbyla KL, Yap VB, Pahwa A, Shao Y, Huttley GA (2013) {The embedding problem
  for Markov models of nucleotide substitution}.
\newblock \emph{PLoS ONE} 8:e69187.

\bibitem{jia2014overshoot}
Jia C, Qian M, Jiang D (2014) {Overshoot in biological systems modelled by
  Markov chains: a non-equilibrium dynamic phenomenon}.
\newblock \emph{IET Syst Biol} 8:138--145.

\bibitem{guerry2014some}
Guerry MA (2014) {Some results on the embeddable problem for discrete-Time
  Markov models in manpower planning}.
\newblock \emph{Commun Stat-Theor M} 43:1575--1584.

\bibitem{norris1998markov}
Norris JR (1998) \emph{{Markov Chains}} (Cambridge University Press,
  Cambridge).

\bibitem{qian2007phosphorylation}
Qian H (2007) Phosphorylation energy hypothesis: open chemical systems and
  their biological functions.
\newblock \emph{Annu Rev Phys Chem} 58:113--142.

\bibitem{cornish2012fundamentals}
Cornish-Bowden A (2012) \emph{{Fundamentals of Enzyme Kinetics}}
  (Wiley-Blackwell, Weinheim), 4th edition.

\bibitem{sakmann2009single}
Sakmann B, Neher E (2009) \emph{{Single-channel Recording}} (Springer-Verlag,
  New York), 2nd edition.

\bibitem{alberty2004principle}
Alberty RA (2004) Principle of detailed balance in kinetics.
\newblock \emph{J Chem Educ} 81:1206.

\bibitem{berman1979nonnegative}
Berman A, Plemmons RJ (1979) \emph{{Nonnegative Matrices in the Mathematical
  Sciences}} (Academic Press, New York).

\bibitem{kelly2011reversibility}
Kelly FP (2011) \emph{{Reversibility and Stochastic Networks}} (Cambridge
  University Press, Cambridge).

\end{thebibliography}
\end{document}